\documentclass[letterpaper, 10 pt, conference]{ieeeconf}  % Comment this line out if you need a4paper

\IEEEoverridecommandlockouts                              % This command is only needed if 
\usepackage{graphics} % for pdf, bitmapped graphics files
\usepackage{epsfig} % for postscript graphics files
\usepackage{mathptmx} % assumes new font selection scheme installed
\usepackage{times} % assumes new font selection scheme installed
\usepackage{amsmath} % assumes amsmath package installed
\usepackage{amssymb}  % assumes amsmath package installed
\providecommand{\norm}[1]{\lVert#1\rVert}
\newtheorem{theorem}{Theorem}
\newtheorem{assumption}{Assumption}%[theorem]
\newtheorem{definition}{Definition}%
\newtheorem{lemma}{Lemma}%
\usepackage{xcolor}

\usepackage{color}
\newcommand{\red}{\textcolor{black}}
\title{\LARGE \bf
Modified projected Gauss-Newton method for constrained nonlinear least-squares: application to power flow analysis
}

\author{Yassine Nabou$^{1}$, Lucian Toma$^{2}$ and Ion Necoara$^{1,3}$% <-this % stops a space
\thanks{$^{1}$Automatic Control and Systems Engineering Department, University Politehnica Bucharest, 060042 Bucharest, Romania. 
        {\tt\small yassine.nabou@stud.acs.upb.ro; ion.necoara@upb.ro}}%
\thanks{$^{2}$Electrical Power Systems Department, University Politehnica Bucharest, 060042 Bucharest, Romania. 
        {\tt\small  lucian.toma@upb.ro}. }%        
\thanks{$^{3}$Gheorghe Mihoc-Caius Iacob Institute of Mathematical Statistics and Applied Mathematics of the Romanian Academy, 050711 Bucharest, Romania.} 
        }
\begin{document}

\maketitle
\thispagestyle{empty}
\pagestyle{empty}

%%%%%%%%%%%%%%%%%%%%%%%%%%%%%%%%%%%%%%%%%%%%%%%%%%%%%%%%%%%%%%%%%%%%%%%%%%%%%%%%
\begin{abstract}
In this paper, we consider a modified projected Gauss-Newton method for solving constrained nonlinear least-squares problems. We assume that the functional constraints are smooth and the the other constraints are represented by a simple closed convex set.  We formulate the nonlinear least-squares problem as an optimization problem using the Euclidean norm as a merit function. In our method, at each iteration we linearize the functional constraints inside the merit function at the current point and add a quadratic regularization, yielding a strongly convex subproblem that is easy to solve, whose solution is the next iterate.     We present global convergence guarantees for the proposed method under mild assumptions. In particular, we prove stationary point convergence guarantees  and under Kurdyka-Lojasiewicz (KL) property for the objective function we derive convergence rates depending on the KL parameter. Finally, we show the efficiency of this method on the power flow analysis problem using several IEEE bus test cases.
\end{abstract}

%%%%%%%%%%%%%%%%%%%%%%%%%%%%%%%%%%%%%%%%%%%%%%%%%%%%%%%%%%%%%%%%%%%%%%%%%%%%%%%%

\section{INTRODUCTION}
\noindent In many areas of engineering, such as maximum likelihood estimations, non-linear data fitting, parameter estimation or power flow analysis, one finds applications that can be recast as  nonlinear least-squares problems of the form \cite{JeRo:69,Har:61,CoWeZh:20}: 
\begin{align}\label{eq:non_lin_pb}
&\min \|F(x)\|\\
&s.t.\; x\in\mathbf{C}\subseteq \mathbb{R}^n,\nonumber 
\end{align}
where $\mathbf{C}$ is a closed convex set and $F = (F_1,\cdots,F_m)$, $F_i:\mathbb{R}^n\to \mathbb{R}$ for $i=1:m$, are nonlinear differentiable functions. When $\mathbf{C} = \mathbb{R}^n$ and $m=n$,  problem \eqref{eq:non_lin_pb} is equivalent to a squared system of nonlinear equations. Hence several algorithms were proposed for solving this problem, among these algorithms  the most popular is Newton-Raphson method (NR) \cite{Sto:74}. In Newton-Raphson method one uses the inverse of the Jacobian matrix in order to update the iterations, i.e., the iterations are of the following form:
\begin{align*}
x^{+} = x - \nabla F(x)^{-1}F(x),
\end{align*} 
where $x$ is the current iteration and $\nabla F(x)$ is the Jacobian matrix of $F(x)$. Although NR has  fast convergence, it has several  drawbacks. First of all, it can happen that at current test point the Jacobian is degenerate; in this case the method is not well-defined. Secondly, this convergence is not guaranteed when the initial point $x_0$ is far from the optimum \cite{Mil:09}. Many approaches have been proposed in order to deal with these challenges, e.g., improving the starting point \cite{BrCaMu:20}, or using different approximations for the Jacobian \cite{ChSh:06,AbCuFl:14}. In \cite{Nes:07}, Nesterov proposed a modified Gauss-Newton scheme (M-GN) for solving unconstrained nonlinear least-squares problems. The M-GN method constructs a convex model by linearizing the nonlinear function $F$ inside a sharp merit function and adding a quadratic regularization term, i.e.:
\begin{align*}
x^{+} = \arg\min\limits_{y \in\mathbb{R}^n} \|F(x) + \nabla F(x)(y - x)\| + \frac{M}{2}\|y - x\|^2.
\end{align*}
When $M = 0$, we recover the NR method described above. In  \cite{Nes:07} it was  proved that, under a nondegenerate assumption (i.e., $\sigma_{\text{min}}(\nabla F(x)) >0$ for all $x$ in the level set of $\|F(x_0)\|$, where $x_0$ is the starting point and $\sigma_{\text{min}}$ denotes the smallest singular value), this scheme has global convergence. Moreover, the solution  of each subproblem can be computed with a standard convex optimization solver. Further, problem \eqref{eq:non_lin_pb} is equivalent to the following composite optimization problem:
\begin{align}\label{eq:op-form1}
\min\limits_{x\in\mathbb{R}^n} \| F(x) \|^2  + \red{I}_{\mathbf{C}}(x),
\end{align}
\red{where $I_{\mathbf{C}}$ is the indicator function of the convex set $\mathbf{C}$}. Note that using only the norm $\|\cdot\|$ as the merit function is beneficial than using $\|\cdot\|^2$, since in the latest case the condition number is doubled. Another possible algorithm for solving this problem is the Projected Gradient Descent (PGD) \cite{Nes:13,HaBoH:16,SaBrAi1:90,SaBrAi2:90}. The standard PGD algorithm is given by:
\begin{align*}
x^{+} = \Pi_{\mathbf{C}}\left(x - \alpha \nabla F(x) F(x)\right),
\end{align*}
where $\Pi_{\mathbf{C}}$ is the projection operator \red{(see Section \ref{sec:N_P})}  and $\alpha$ is a step size. PGD descent is a simple method easy to implement, but the main drawback is that it  has  slow convergence.
\medskip

\noindent A natural questions arises whether  we can prove global convergence of MG-N method without assuming the nondegeneracy assumption on the Jacobian $\nabla F(x)$ , i.e., without assuming  $\sigma_{\text{min}}(\nabla F(x)) >0 $ for all $x$ in the level set of $\|F(x_0)\|$ (see \eqref{eq:levset}). Such a condition is conservative and it \red{is not} always satisfied in practice. In this paper we answer positively to this question, i.e., we consider a  Modified Projected Gauss-Newton method (MPG-N) for solving problem \eqref{eq:non_lin_pb}, where $\mathbf{C}$ is a simple closed  convex set. At each iteration, MPG-N aims to solve the following strongly convex subproblem:
\begin{align}\label{eq:op-prb}
x_{k+1} \!=\! \arg\min\limits_{x\in\mathbf{C}} \|F(x_k) \!+\! \nabla F(x_k)(x \!-\! x_k)\| \!+\! \frac{M}{2}\|x - x_k\|^2,
\end{align} 
which is a slightly modified version of \red{\cite{Nes:07}} as it considers constraints $x\in\mathbf{C}$. We prove, under mild assumptions, that this scheme can achieve global convergence  without any assumption on the Jacobian matrix. More precisely, we prove that any limit point of the sequence generated by MPG-N is a stationary point and under the  Kurdyka-Lojasiewicz (KL) property, we derive convergence rates in function value depending on the KL parameter. Finally, we consider solving a  power flow analysis problem, with functional constraints which  \red{do not} usually satisfy the non-degenerate assumption, while it satisfies the KL property. We compare the performance of such a scheme with the projected gradient scheme and demonstrate its efficiency of the proposed method on several  IEEE bus test cases.\\

\noindent \textit{Content}. The rest of the paper is organized as follows: Section \ref{sec:N_P} provides some notations and preliminaries, Section \ref{sec:ALG}  presents the new  algorithm and the convergence results, Section \ref{sec:PFA_SIM} describes the  power flow analysis problem  and  numerical results on several  IEEE bus test cases.

%where $h$ is proper lower semicontinuous function.  
%Cubic Newton method CNM is a variant of Newton method that has global convergence for unconstrained minimization of twice continuously differentiable function \cite{NesPol:06}. At each iteration of the CNM one needs to minimize a cubic model that consists of a third-order regularization of the second-order Taylor approximation of the objective function $f$. For a given tolerance $\epsilon >0$, Nesterov and Polyak \cite{NesPol:06} proved that the CNM takes at most $\mathcal{O}(k^{-3/2})$ iterations to generate an -approximate stationary point of the objective function (i.e., an iterate $x_k$ such that $\|\nabla f(x_k) \| \leq \epsilon$ ), when $f$ is a nonconvex function with Lipschitz continuous Hessian. This is faster than the usual gradient descent that achieve $\mathcal{O}(k^{-1/2})$.  Although, the requirements for the applicability of these methods are different since it necessitate computing the Hessian, however, there exist numerous problems, where computation of the Hessian is not much harder than computation of the gradient, and the iteration costs of both methods are comparable. Thus, it is convenient to use the CNM for solving the PF problems, which can guarantees global convergence and performs well when the problems are ill-conditioned.

\section{Notations and preliminaries}\label{sec:N_P}
\noindent We denote a finite-dimensional real vector space with $\mathbb{E}$ and by $\mathbb{E}^{*}$ its dual space composed of linear functions on $\mathbb{E}$. Using a self-adjoint positive-definite operator $D:\mathbb{E}\rightarrow \mathbb{E}^{*}$(notation $D=D^{*}\succ 0$), we can endow these spaces with conjugate Euclidean norms:
\begin{align*}
\norm{x}=\langle Dx,x\rangle\red{^{\frac{1}{2}}},\quad x\in \mathbb{E},\qquad \norm{g}_{*}=\langle g,D^{-1}g\rangle^{\frac{1}{2}},\quad g\in \mathbb{E}^{*}.
\end{align*}
For simplicity, we consider in the following $\mathbb{E} = \mathbb{R}^n$ and $D$ is the identity matrix.
Let $F = (F_1,\cdots,F_m)$, where $F_i$'s, $i=1:m$, are differentiable functions and the Jacobian is Lipschitz continuous, i.e.:
\begin{align*}
\|\nabla F(x) - \nabla F(y)\|\leq L_{F} \|x - y\| \;\;  \forall x,y\in\mathbb{R}^n. 
\end{align*}
It follows that \cite{Nes:07}:
\begin{align}\label{eq:lipsh_descn}
\|F(x)\| \!\!-\!\! \|F(y) \!+\! \nabla F(y) (x \!\!-\!\! y)\|\leq \frac{L_F}{2}\|x \!\!-\!\! y\|^2 \;\; \forall x,y\!\in\!\mathbb{R}^n.
\end{align}
Let $h$ be proper lower semicontinuous function and $\mu>0$. Then, the proximal operator with respect to $h$ is:
\begin{align*}
\text{prox}_{\mu h}(x) = \arg\min_{y} h(y) + \frac{\mu}{2}\| y - x \|^2,  
\end{align*}
and the Moreau envelop is defined as:
\begin{align*}
h_{\mu}(x) = \min_{y}\; h(y) + \frac{\mu}{2}\|y - x\|^2. 
\end{align*}
When $h$ is the indicator function of a convex set $C$, $I_{C}$, then the proximal operator is the projection:
\begin{align*}
\text{prox}_{\mu \red{I}_{\mathbf{C}}}(x) & = \Pi_{\textbf{C}}(x) = \arg\min_{y\in C} \|y - x\|^2.
\end{align*}
We say that $h$ is $\mu$-weakly convex if the function $$x \mapsto h(x) + \frac{\mu}{2}\|x\|^2 $$ is convex. The level set of $h$ at $x_0$ is defined:
\begin{align}\label{eq:levset}
\mathcal{L}(h(x_0)): = \lbrace x\in\mathbb{R}^n: h(x)\leq h(x_0) \rbrace.    
\end{align}
Next, we provide few definitions and properties concerning subdifferential calculs (see also \cite{Mor:06,Rock:98}).\\

\begin{definition}
\textit{(Subdifferential)}: Let $f: \mathbb{R}^n \to \bar{\mathbb{R}}$ be a proper lower semicontinuous function. For a given $x \in \text{dom} \; f$, the Frechet subdifferential of $f$ at $x$, written $\widehat{\partial}f(x)$, is the set of all vectors $g_{x}\in\mathbb{R}^{n}$ satisfying:
\begin{equation*}
\lim\limits_{x\neq y,y\to x}\frac{f(y) - f(x) - \langle g_{x}, y - x\rangle}{\norm{x-y}}\geq 0.
\end{equation*}
When $x \notin\text{dom} \; f$, we set $\widehat{\partial} f(x) = \emptyset$. The limiting-subdifferential, or simply the subdifferential, of $f$ at $x\in \text{dom} \, f$, written $\partial f(x)$, is defined through the following closure process \cite{Mor:06}:
\begin{align*}
\partial f(x):= &\left\{ g_{x}\in \mathbb{E}^{*}\!\!: \exists x^{k}\to x \;\text{with}\; f(x^{k})\to f(x) \;\right.\\
 &\quad \left. \text{and} \; \exists g_{x}^{k}\in\widehat{\partial} f(x^{k}) \;\; \text{with} \;\;  g_{x}^{k} \to g_{x}\right\}. 
\end{align*} 
\end{definition}

\noindent Note that we have $\widehat{\partial}f(x)\subseteq\partial f(x)$  for each $x\in\text{dom}\,f$. In the previous inclusion, the first set is closed and convex while the second one is closed, see e.g., \cite{Rock:98}(Theorem~8.6).   
\noindent For any $x\in \text{dom} \; f$ let us define:
\begin{align*}
\red{S_f(x)} = \text{dist}\big(0,\partial f(x)\big):=\inf\limits_{g_{x}\in\partial f(x)}\norm{g_{x}}.
\end{align*}
\noindent If $\partial f(x) = \emptyset$, we set $S_f(x) = \infty$. 
Let us also recall the definition of a function satisfying the \textit{Kurdyka-Lojasiewicz (KL)} property (see \cite{BolDan:07} for more details).

\medskip 

\begin{definition}
\label{def:kl}
\noindent A proper lower semicontinuous  function $f: \mathbb{R}^n\rightarrow \bar{\mathbb{R}}$ satisfies  \textit{Kurdyka-Lojasiewicz (KL)} property on the compact set $\Omega \subseteq \text{dom} \; f$ on which $f$ takes a constant value $f_*$ if there exist $\delta, \epsilon, q >0$ such that   one has:
\begin{align}\label{eq:kl}
&f(x) - f_*  \leq \sigma_q\; \red{S_f(x)^q} \\ \nonumber
&  \forall x\!: \;  \text{dist}(x, \Omega) \leq \delta, \; f_* < f(x) < f_* + \epsilon.  
\end{align}
\end{definition}  

\medskip

\noindent  Note that the relevant aspect of the KL property is when $\Omega$ is a subset of critical points for $f$, i.e.  $\Omega \subseteq \{x: 0 \in \partial f(x) \}$, since it is easy to establish the KL property when $\Omega$ is not related to critical points. The KL property holds for a large class of functions including semi-algebraic functions (e.g., real polynomial functions), vector or matrix (semi)norms (e.g., $\|\cdot\|_p$ with $p \geq 0$ rational number), trigonometric functions, logarithm functions,  exponential functions and  uniformly convex functions,  see \cite{BolDan:07} for a comprehensive list. \\

%--------------------------------------------------------------------------

\section{Modified Projected Gauss-Newton method}\label{sec:ALG}
\noindent In this section, we present the Modified Projected Gauss-Newton  (MPG-N) method and then derive convergence results. We recall the problem of our interest is:
\begin{align}\label{eq:op_prb_ours}
\min\limits_{x\in\mathbb{R}^n} f(x) : = \|F(x)\| + I_{\mathbf{C}}(x).
\end{align}
We consider the following assumption:
\begin{assumption}\label{eq:ass1}
\begin{enumerate}
\item $F$ is differentiable and the Jacobian is Lipschitz continuous:
\begin{align*}
\| \nabla F(x) - \nabla F(y) \|\leq L_F\|x - y\|, \quad \forall x,y\in\mathbf{C}.
\end{align*}
\item Problem \eqref{eq:op-prb} has solution, i.e., there exist $x^* \in\mathbf{C}$ such that $f(x^*) >  -\infty$.
\end{enumerate}
\end{assumption}
An immediate consequence of \eqref{eq:lipsh_descn} is:
\begin{align*}
f(x)\leq \|F(y) + \nabla F(y)(x-y)\| + \frac{L_F}{2}\| x-y \|^2\quad \forall x,y \in\mathbf{C}.
\end{align*}
Then, for $M > 0$, we define the modified projected Gauss-Newton iterate at a point $x\in\mathbf{C}$ as follows:
\begin{align}\label{eq:iter}
T_M (x) & = \arg\min_{y\in\mathbf{C}} \Psi_{M}(y;x)\\
        &:= \arg\min_{y\in\mathbf{C}} \|F(x) + \nabla F(x)(y-x)\| + \frac{M}{2}\| y-x \|^2.\nonumber
\end{align}
Note that this subproblem is strongly convex, hence $T_M (x)$ is well defined and unique. Finally, the modified projected Gauss-Newton algorithm is as follows: 
\begin{center}
\noindent\fbox{%
\parbox{8.4cm}{%
\textbf{MPG-N algorithm}\\
Chose $x_{0}\in\mathbf{C}$ and  $\red{L_0,\delta >0}$. For $k\geq 0$ do:\\
%\begin{itemize}
%\item[1-] 
Find $L_0 \leq M_k \leq 2L_F$ such that:
\begin{align}\label{eq:alg_desc}
 \frac{\delta}{2}\| T_{M_k}(x_k) - x_k\|^2\leq  \Psi_{M_k}\left(T_{M_k}(x_k);x_k\right) \!-\! f\left(T_{M_k}(x_k)\right)
\end{align}
Update  $x_{k+1} = T_{M_k}(x_k)$.
%\end{itemize}

%
}
}
\end{center}

\medskip 

\noindent The first step of MPG-N algorithm consists of finding a constant $M_k >0$ such that inequality \eqref{eq:alg_desc} holds. If the constant $L_F$ is known, we can take $M_k = L_F + \delta$. Otherwise, we can apply the following line search procedure \cite{NesPol:06}:
\begin{align*}
    &\textbf{While}\; \eqref{eq:alg_desc}\; \text{is not satisfied}\; \textbf{do}\;\; M_k = 2M_k \;\\
    &M_{k+1} = \text{max}\left(\frac{M_k}{2} , L_0\right).
\end{align*}

\noindent The next lemma shows that this process is well defined.
\begin{lemma}\label{eq:lem1}
Let Assumption \ref{eq:ass1} hold. At $k${th} iteration of MPG-N algorithm, if $M_k - L_F \geq \delta$, then inequality \eqref{eq:alg_desc} holds.
\end{lemma} 
\begin{proof}
%Since $T_{M_k}(x_k)$ is the global solution of the subproblem \eqref{eq:iter}, then we have:
%\begin{align*}
%&\|F(x_k) + \nabla F(x_k)(T_{M_k}(x_k) - x_k)\|  + \frac{M_k}{2}\|T_{M_k}(x_k) - x_k\|^2 \\
%&\leq \min_{x\in\mathbf{C}}\|F(x_k) + \nabla F(x_k)(x - x_k)\|  + \frac{M_k}{2}\|x - x_k\|^2 \leq f(x_k). 
%\end{align*}
We have from inequality \eqref{eq:lipsh_descn} that:
\begin{align*}
&\frac{M_k - L_F}{2}\|T_{M_k}(x_k) - x_k\|^2 + f(T_{M_k}(x_k))\\
&\leq \|F(x_k) + \nabla F(x_k)(T_{M_k}(x_k) - x_k)\|  + \frac{M_k}{2}\|T_{M_k}(x_k)- x_k\|^2\\
&= \Psi_{M_k}\left(T_{M_k}(x_k);x_k\right).
\end{align*}
Since  $M_k - L_F \geq \delta$, it follows immediately that:
 \begin{align*}
 \frac{\delta}{2}\| T_{M_k}(x_k) - x_k\|^2\leq  \Psi_{M_k}\left(T_{M_k}(x_k);x_k\right) - f(T_{M_k}(x_k)) .
\end{align*}
Hence, this is the statement of the lemma.
\end{proof}
Note that Lemma \ref{eq:lem1} ensures that \eqref{eq:alg_desc} always holds, provided that $M_k \geq L_f + \delta$. However,  
in practice, using the line search procedure allows us to work   with $M_k$ small (i.e.,  $M_k \leq L_F$) such that condition \eqref{eq:alg_desc} holds. 
Next, let us discuss the solution of the subproblem \eqref{eq:iter}. Following \cite{Nes:13}, we have:
\begin{align*}
&\min_{y\in \mathbf{C} } \|F(x) + \nabla F(x)(y-x)\| + \frac{M}{2}\| y-x \|^2\\
&= \min_{y\in \mathbf{C}} \max_{\|s\|\leq 1} \langle s,F(x) + \nabla F(x)(y-x)\rangle + \frac{M}{2}\| y-x \|^2\\
&=  \max_{\|s\|\leq 1} \min_{y\in \mathbf{C}} \langle s,F(x) + \nabla F(x)(y-x)\rangle + \frac{M}{2}\| y-x \|^2\\
&=  \max_{\|s\|\leq 1} \min_{y\in \mathbf{C}}  \langle \nabla F(x)^{T}s,(y-x)\rangle + \frac{M}{2}\| y-x \|^2 + \langle s,F(x) \rangle \\
&=  \max_{\|s\|\leq 1} \min_{y\in \mathbf{C}}  \frac{M}{2}\| y-x + \frac{1}{M}\nabla F(x)^{T}s \|^2 - \frac{1}{2M}\|F(x)^{T}s\|^2 \\
& \quad + \langle s,F(x) \rangle \\
&=  \max_{\|s\|\leq 1}  \frac{M}{2}\|  \Pi_{\textbf{C}}\left(x - \frac{1}{M}\nabla F(x)^{T}s\right)-x + \frac{1}{M}\nabla F(x)^{T}s \|^2  \\
& \quad - \frac{1}{2M}\|F(x)^{T}s\|^2 + \langle s,F(x) \rangle,
%&=  \max_{\|s\|\leq 1} \text{dist}({\mathbf{C}},x - \nabla F(x)^{T}s)- \frac{1}{2M}\|F(x)^{T}s\|^2 + \langle s,F(x) \rangle. \\
\end{align*}
which can be solved with standard convex optimization tools, such as trust-region methods \cite{CoGoTo:00}. 

%%%%%%%%%%

\subsection{Convergence analysis}
\noindent In this section we derive convergence results for MPG-N algorithm. First, we  can prove the following descent:
\begin{lemma}
Let Assumption \ref{eq:ass1} hold. Let $(x_{k})_{k\geq 0}$ be generated by MPG-N algorithm. Then, we have:
\begin{enumerate}
\item  Sequence $(f(x_k))_{k\geq 0}$ is nonincreasing and satisfies: 
\begin{align}\label{eq:desc_fct}
\frac{\delta}{2}\| x_{k+1} - x_k\|^2\leq f(x_k) - f(x_{k+1}).
\end{align}
\item The sequence $(x_{k})_{k\geq 0}$ satisfies:
\begin{align*}
\sum_{k=1}^{\infty}\| x_{k+1} - x_k\|^2 <\infty,\quad \lim_{k\to \infty} \| x_{k+1} - x_k\|^2 = 0.
\end{align*}
\end{enumerate}
\end{lemma}
\medskip 
\begin{proof}
We have:
\begin{align*}
\Psi_{M_k}\left(T_{M_k}(x_k);x_k\right)\leq \min\limits_{x\in\mathbf{C}}\Psi_{M_k}\left(x;x_k\right)\leq \Psi_{M_k}\left(x_k;x_k\right) = f(x_k).
\end{align*}
Then, combining this inequality with equation \eqref{eq:alg_desc} we get the first statement. Further, summing up the inequality \eqref{eq:alg_desc} and using that $f$ is bounded from below by $f^*$, we get:
\begin{align*}
\sum_{k=1}^{N}\frac{\delta}{2}\|x_{k+1} - x_k\|^2 \leq f(x_0) - f(x_{N})\leq  f(x_0) - f^*,
\end{align*}
and the second statement follows.
\end{proof}

\medskip 

\noindent In \cite{DruPaq:19,NabNec:22}, the authors prove that for the composite problem \eqref{eq:op-prb}, the quantity $\text{dist}(0,\partial f(x_{k+1}))$ \red{does not} always tend to zero in the limit, even if $\|x_{k+1} - x_k\|$ goes to zero. Thus, we must look elsewhere for a connection between $\text{dist}(0,\partial f(\cdot))$ and $\|x_{k+1} - x_k\|$. Let us start with the following observation, whose proof can be found in Lemma 4.2 \cite{DruPaq:19}: the function $f(x) := \|F(x)\| + \red{I}_{\mathbf{C}}(x)$ is $L_F$-weakly convex. Weak convexity of $f$ has an immediate consequence on the Moreau envelope, denoted  $f_\mu$:   
\medskip 
\begin{lemma}(Lemma 4.3 \cite{DruPaq:19})
Let  $ \mu > L_F $. Then, the proximal map $\text{prox}_{\mu f}$ is well-defined and single-valued. The Moreau envelope $f_\mu$ is smooth with gradient given by:
\begin{align*}
\nabla f_\mu (x) = \mu(x - \text{prox}_{\mu f}(x)).
\end{align*}
\end{lemma}
Further, we have the following lemma whose prove is similar to the proof of Lemma $5$  in \cite{NabNec:22}.

\medskip 
\begin{lemma}\label{th:2}
Let Assumption \ref{eq:ass1} holds. Let $(x_k)_{k\geq 0}$ be generated by MPG-N method and consider $y_{k+1} = \text{prox}_{\mu f}(x_k)$, where $\frac{1}{\mu}\in (0,\frac{1}{3L_F})$. Then, we have the following relations:
\begin{enumerate}
\item $\|y_{k+1} - x_{k}\|^2 \leq \frac{\mu}{\mu - 3 L_F} \|x_{k+1} - x_k\|^2$\\
\item $\text{dist}(0,\partial f(y_{k+1})) \leq \mu \|y_{k+1} - x_k\|$.
\end{enumerate}
\end{lemma}
\medskip 
\begin{proof}
Let us prove the first statement. Since $f$ is weakly convex, then $y_{k+1}$ is well-defined and unique. Thus:
\begin{align}\label{eq:desc_yk}
f(y_{k+1}) + \frac{\mu}{2}\|y_{k+1} - x_k\|^2\leq f(x_{k+1}) + \frac{\mu}{2}\|x_{k+1} - x_k\|^2.
\end{align} 
Further, from the definition of $x_{k+1}$, we have:
\begin{align*}
f(x_{k+1}) &\stackrel{\eqref{eq:alg_desc}}{\leq} \|F(x_k) \!+\! \nabla F(x_k) (x_{k+1} \!\!-\!\! x_k)\| + \frac{M_k}{2}\|x_{k+1} \!\!-\!\! x_k\|^2\\ \nonumber
&\stackrel{\eqref{eq:iter}}{\leq}\min_{x\in\mathbf{C}}\|F(x_k) + \nabla F(x_k) (x - x_k)\| + \frac{M_k}{2}\|x - x_k\|^2\\ 
&\stackrel{\red{\eqref{eq:alg_desc}}}{\leq} \min_{x\in\mathbf{C}} f(x) + \frac{M_k + L_F}{2}\|x - x_k\|^2\\ 
&\leq f(y_{k+1}) + \frac{M_k + L_F}{2}\|y_{k+1} - x_k\|^2,
%&\leq \min_{x\in\mathbf{C}} f(x) + \frac{M_k + L_F}{2}\|x - x_k\|^2.
\end{align*}
\red{where the last inequality follows by taking $x = y_{k+1}$}. Thus, we have:
\begin{align}\label{eq:01}
f(x_{k+1})&\leq f(y_{k+1}) + \frac{3 L_F}{2}\|y_{k+1} - x_k\|^2.
\end{align}
Finally, combining this inequality with \eqref{eq:desc_yk}, we get:
\begin{align*}
\|y_{k+1} - x_k\|^2 \leq \frac{\mu}{\mu -  3L_F}\|x_{k+1} - x_k\|^2,
\end{align*}
which proves the first statement. Further, from the optimality conditions of $y_{k+1}$, we get:
\begin{align}\label{eq:subg}
-\mu(y_{k+1} - x_k)\in\partial f(y_{k+1}). 
\end{align}
Thus, the second statement follows.  
\end{proof}
Using the strict descent and Lemma \ref{th:2}, we can conclude the following global convergence rate:
\medskip 
\begin{theorem}
Let the assumptions of Lemma \ref{th:2} hold. Then:
\begin{align*}
\min\limits_{j=1:k}\text{dist}(0,\partial f(y_j))\leq \mathcal{O}\left(\frac{1}{k^{1/2}}\right).
\end{align*}
Moreover, any limit point of the sequence $(x_k)_{k\geq 0}$ is a stationary point of problem \eqref{eq:op_prb_ours}.
\end{theorem}
\begin{proof}
From \red{Lemma \ref{th:2}}, we have:
\begin{align*}
\text{dist}(0,f(y_{k+1}))^2&\leq \frac{\mu^3}{\mu - 3L_f}\|x_{k+1} - x_k\|^2.
\end{align*}
Further, combining this inequality with \eqref{eq:desc_fct}, we get:
\begin{align*}
\text{dist}(0,f(y_{k+1}))^2\leq \frac{6\mu^3}{\delta(\mu -  3L_F)} f(x_k) - f(x_{k+1}).
\end{align*}
Summing up this inequality and taking the minimum we get:
\begin{align*}
\min\limits_{j=0:k}\text{dist}(0,f(y_{j+1}))\leq \sqrt{\frac{6\mu^3}{\delta(\mu -  3L_F)}}\frac{1}{k^{1/2}}.
\end{align*}
which prove our first statement. Further, let $x^*$ be a limit point of $(x_k)_{k\geq 0}$, then one can notice that it is also a limit point of the sequence $(y_k)_{k\geq 0}$. This means that there exist a subsequence $(y_{k_j})_{j\geq 0}$ such that $ y_{k_j} \to x^*$. Since $F$ is continuous, then $f(y_{k_j}) \to f(x^*)$. Note that we have $\mu(y_{k_j } - x_{k_j - 1}) \in\partial f(y_{k_j})$ and $(y_{k_j} - x_{k_j - 1}) \to 0$. Then, we conclude from the definition of the generalized subgradient that $0\in\partial f(x^*)$ and hence $x^*$ is a stationary point.   
\end{proof}

%Let us recall the following two lemma where the proof can be found in \cite{NabNec:22}.
%\begin{lemma}\label{rmq1}
%Let $\left(x_{k}\right)_{k\geq 0}$ generated by Algorithm GCHO be bounded and $y_{k} = \text{prox}_{\mu f}(x_k)$ where $\mu\in (0,\frac{1}{L_F + M_k})$. Then, the set of limit points of the sequence $\left(y_{k}\right)_{k\geq 0}$ coincides with the set of limit points of $\left( x_{k}\right)_{k\geq 0}$.
%\end{lemma}

\subsection{Better rates under KL}
\noindent In \cite{Nes:07}, the authors impose a nondegeneracy assumption on the Jacobian, that is, $\sigma_{\text{min}}(\nabla F(x)) >0$ for all $x$ in the level set of $\|F(x_0)\|$ in order to prove global convergence rate for MG-N method. Such a condition is  not always valid in practice. In this section, we derive improved convergence rates for MPG-N method provided that the objective function satisfies the KL property. In general, the KL condition is less conservative than the nondegeneracy condition (see Section \ref{sec:PFA_SIM}).  Let us denote the set of limit points of $(x_{k})_{k\geq 0}$ by:
\begin{align*}
\Omega (x_{0})=&\lbrace \bar{x}\in \mathbb{E}: \exists \text{ an increasing sequence of integers }\\
& (k_{t})_{t\geq0}, \text{ such that } x_{k_{t}}\to \bar{x} \text{ as } t\to \infty \rbrace.
\end{align*} 
We have the following convergence rate:

\medskip 

\begin{theorem}\label{Thm:kl}
Let the assumptions of Lemma \ref{th:2} hold. Additionally, assume that $f$ satisfy the KL property \eqref{eq:kl} on $\Omega(x_0)$. Then, the following convergence rates hold for the sequence $(x_{k})_{k\geq 0}$ generated by MPG-N algorithm in function values:
\begin{enumerate}
\item[$\bullet$]If $q\geq 2$, then $f(x_{k})$ converge to $f_{*}$ linearly for $k$ sufficiently large.
\item[$\bullet$]If $q < 2$, then $f(x_{k})$ converge to $f_{*}$ at sublinear rate of order $\mathcal{O}\left(\frac{1}{k^{\frac{q}{2 - q}}}\right)$ for $k$ sufficiently large.
\end{enumerate}
\end{theorem}

\medskip 

\begin{proof}
\red{We have:}
\begin{align*}
f(x_{k+1}) - f_* &\stackrel{\red{\eqref{eq:01}}}{\leq} f(y_{k+1}) - f_* + \frac{3L_F}{2}\|y_{k+1} - x_k\|^2\\
            &\stackrel{\red{\eqref{eq:kl}}}{\leq} \sigma_q \red{S_f(y_{k+1})^q} +\frac{3L_F}{2}\|y_{k+1} - x_k\|^2\\
&\leq \sigma_q \mu^q \|y_{k+1} - x_k\|^q + \frac{3L_F}{2}\|y_{k+1} - x_k\|^2\\
&\leq \sigma_q \mu^q \left(\frac{\mu}{\mu -  3L_F}\right)^{q/2} \|x_{k+1} - x_k\|^q \\
&\quad + \frac{2\mu}{3L_F(\mu - 3L_F)}\|x_{k+1} - x_{k}\|^2\\
&\leq  C_1(f(x_k) - f(x_{k+1}))^\frac{q}{2} +C_2(f(x_{k}) - f(x_{k+1})).    
\end{align*}
where \red{the third and the fourth inequalities follow from Lemma \ref{th:2}, the last inequality follows from the descent \eqref{eq:desc_fct}}, $C_1 = \sigma_q \mu^q \left(\frac{\mu}{\mu -  3L_F}\right)^{q/2}(2/\delta)^{q/2}$ and $C_2 =  \frac{4\mu}{\delta(3L_F)(\mu - 3L_F)} $. Denote $\delta_k = f(x_{k}) - f_*$, then we get:
\begin{align*}
\delta_{k+1}\leq C_1(\delta_k - \delta_{k+1})^{\frac{q}{2}} + C_{2}(\delta_k - \delta_{k+1}).
\end{align*}
Using Lemma 2 in \cite{NabNec:22} with $\theta =\frac{2}{q}$ we get our statement.
\end{proof}

\section{Power flow analysis}\label{sec:PFA_SIM}
\noindent  Power flow problems are ones of the most studied in power systems being an important tool for planning and operation of the electric grid. In this section we consider the particular problem of power flow analysis. This is defined as follows.   Consider a power system with $N$ bus (see e.g., Figure \ref{fig:my_label} for the IEEE 14 bus system). We denote $v_i$, $p_i$ and $q_i$ the  complex voltage, active power and reactive power for the $i$ bus, respectively. Let $Y : = G + jB$ be the admittance matrix and denote $p=(p_1,\cdots,p_N)$, $q = (q_1,\cdots,q_N)$ and $v=(v_1,\cdots,v_N)$. Given a complex load vector $s: = s_R + js_I$, then the power flow analysis problem is to find $v=(v_1,\cdots,v_N)$ such that \cite{CoWeZh:20}:
\begin{align}\label{eq:objft}
  F(v) = s\;;\quad F(v) = p + jq = \text{diag}(vv^{H}Y^{H}), 
\end{align}
 where $(.)^{H}$ is the Hermitian transpose. This problem is equivalent to the following optimization problem:
 \begin{figure}
    \centering
    \includegraphics[scale=0.4]{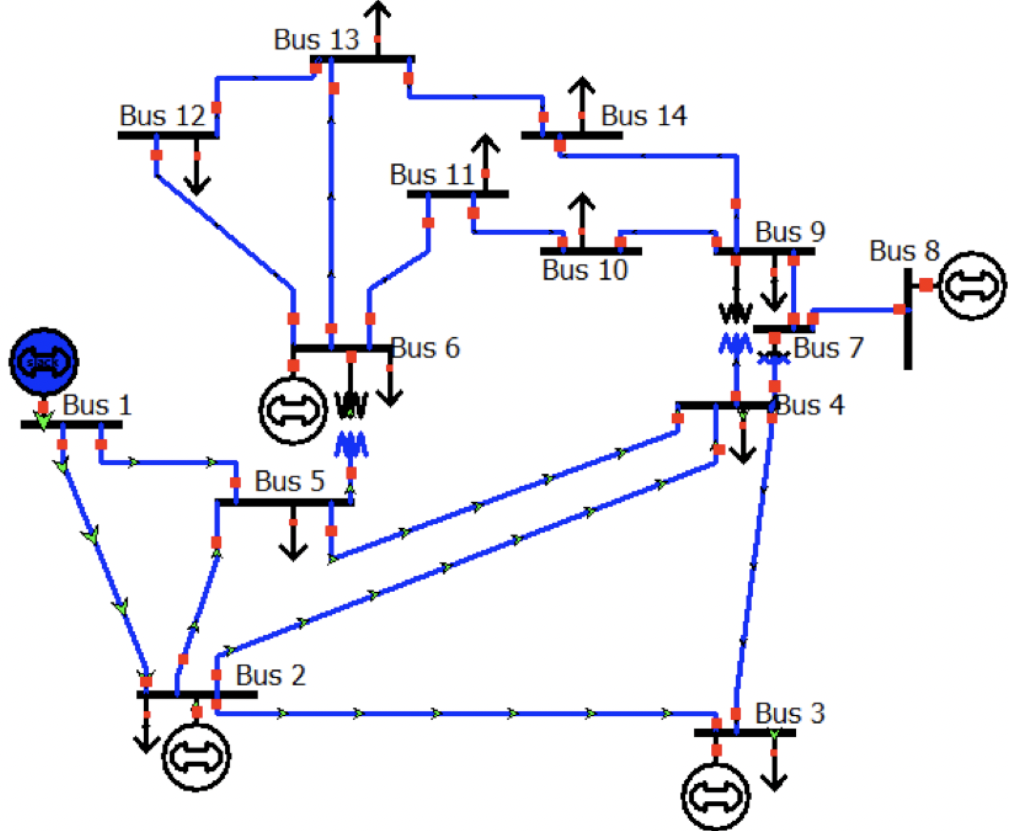}
    \caption{Representation  of the IEEE 14-bus system \cite{ICEG:13}.}
    \label{fig:my_label}
\end{figure}
\begin{align*}
&\min_{v =(u,\theta)} \|F(v) - s\|\\
&s.t. \quad u\in [u_{\text{min}},u_{\text{max}}],\quad \theta\in [-\pi,\pi].
\end{align*}
\noindent In \cite{CoWeZh:20}, the authors provide a similar formulation for the power flow analysis problem, but using $\| \cdot \|^2$ as the merit function to measure the distance between the objective function $F(\cdot)$ and the desired complex load $s$.  As we have mentioned earlier, it is beneficial to use only $\| \cdot\|$ as the merit function. Further, since we have (see e.g., \cite{LeN:97}): 
\begin{align*}
&p_i(u,\theta) = \sum_{k=1}^{N}u_i u_k\left(G(i,k) \text{cos}(\theta_i - \theta_k) + B(i,k) \text{sin}(\theta_i - \theta_k) \right),\\
&q_i(u,\theta) = -\sum_{k=1}^{N}u_i u_k(B(i,k) \text{cos}(\theta_i \!-\! \theta_k) \!+\! G(i,k) \text{sin}(\theta_i \!-\! \theta_k) ),
\end{align*}
and denote: 
$$\textbf{C} = \lbrace (u,\theta) : u\in [u_{\text{min}},u_{\text{max}}],\quad \theta\in [-\pi,\pi] \rbrace,$$
then, the previous optimization problem is equivalent to the following optimization problem:
\begin{align}\label{eq:opt_pr_f}
\min\limits_{x=(u;\theta)\in\textbf{C}} f(x) =
   \begin{Vmatrix}
       p(x) - s_R  \\ q(x) - s_I  
   \end{Vmatrix}.
\end{align}
The most efficient algorithm for solving the (unconstrained) power flow analysis problem is the Newton-Raphson (NR) method \cite{Sto:74}. However it may lead to poor performance when the initialization point is far from the optimum or the system is stressed (i.e., the problem is ill-conditioned). In a recent paper, \cite{CoWeZh:20}, the authors  proposed a hybrid method that combines stochastic gradient descent (SGD)  and the NR methods to overcome the numerical challenges in this problem. The iterative process  starts with the NR algorithm, and if the method detect a divergence (e.g., when the condition number of the Jacobian deteriorates), then switch to the SGD algorithm. After running a few SGD steps, then again switch to the NR iterates and repeat the process until an (approximate) optimal solution is found. Since this hybrid algorithm cannot deal with (simple) constraints as in \eqref{eq:opt_pr_f}, we propose to use our new method,  modified projected Gauss-Newton (MPG-N), and compare its performance  with the projected gradient descent (PGD) method applied to the problem \eqref{eq:op-form1}, where $F$ is given in \eqref{eq:objft}.  In order to apply both methods, one needs to evaluate the gradient of the functions $p(x)$ and $q(x)$. We have the following expressions for the derivatives of $p_i$'s and $q_i$'s:
\begin{align*}
&\frac{\partial p_i}{\partial u_i} \!=\! 2G(i,i) + \!\!\sum_{\substack{k=1 \\ k\neq i}}^{N} u_k\left(G(i,k) \text{cos}(\theta_i \!\!-\!\! \theta_k) \!\!+\!\! B(i,k) \text{sin}(\theta_i \!\!-\!\! \theta_k) \right), \\
&\frac{\partial p_i}{\partial u_k} = u_i\left(G(i,k) \text{cos}(\theta_i \!-\! \theta_k) + B(i,k) \text{sin}(\theta_i \!-\! \theta_k) \right), \forall k \!\neq\! i,\\
&\frac{\partial p_i}{\partial \theta_i} = \sum_{\substack{k=1 \\ k\neq i}}^{N} u_ku_i\left(-G(i,k) \text{sin}(\theta_i \!-\! \theta_k) + B(i,k) \text{cos}(\theta_i \!-\! \theta_k) \right),  \\
&\frac{\partial p_i}{\partial \theta_k} = -u_i u_k\left(-B(i,k) \text{cos}(\theta_i \!\!-\!\! \theta_k) \!\!-\!\! G(i,k) \text{sin}(\theta_i \!\!-\!\! \theta_k)  \right),\! \forall k\neq \!i,\\
&\frac{\partial q_i}{\partial u_i} \!=\!\!-\!2B(i,i) \!\!-\!\!\!\sum_{\substack{k=1 \\ k\neq i}}^{N} u_k\left(B(i,k) \text{cos}(\theta_i \!\!-\!\! \theta_k) \!\!-\!\! G(i,k) \text{sin}(\theta_i \!\!-\!\! \theta_k) \right), \\
&\frac{\partial q_i}{\partial u_k} = -u_i\left(B(i,k)\text{cos}(\theta_i \!-\! \theta_k) - G(i,k)\text{sin}(\theta_i \!-\! \theta_k)\right), \forall k\!\neq\! i,\\
&\frac{\partial q_i}{\partial \theta_i} = \sum_{\substack{k=1 \\ k\neq i}}^{N} u_k u_i\left( B(i,k) \text{sin}(\theta_i \!-\! \theta_k) + G(i,k) \text{cos}(\theta_i \!-\! \theta_k)  \right),   \\
&\frac{\partial q_i}{\partial \theta_k} = -u_k u_i\left( G(i,k) \text{cos}(\theta_i \!-\! \theta_k) + B(i,k) \text{sin}(\theta_i \!-\! \theta_k)  \right),\forall k\!\neq\! i.
\end{align*}
Hence, $\nabla f(x) \in\mathbb{R}^{2N}$ and we have:
\begin{align*}
\nabla f(x) = \sum_{i=1}^{N} \frac{\partial p_i(x)}{\partial x}(p_i(x) - s_R) + \frac{\partial q_i(x)}{\partial x} (q_i(x) - s_I),
\end{align*}
where $\frac{\partial p_i(x)}{\partial x} = \left(\frac{\partial p_i(x)}{\partial u_1};\cdots;\frac{\partial p_i(x)}{\partial u_N};\frac{\partial p_i(x)}{\partial \theta_1};\cdots;\frac{\partial p_i(x)}{\partial \theta_N}\right)$ and $\frac{\partial q_i(x)}{\partial x} = \left(\frac{\partial q_i(x)}{\partial u_1};\cdots;\frac{\partial q_i(x)}{\partial u_N};\frac{\partial q_i(x)}{\partial \theta_1};\cdots;\frac{\partial q_i(x)}{\partial \theta_N}\right)$ for $i=1:N$. Note that the Jacobian $\nabla F$ may be ill-conditioned, but the objective function $f$ (\red{may}) satisfy KL inequality. %\red{since the expressions of $p(\cdot)$ and $q(\cdot)$ are semi-algebraic functions \cite{BolDan:07}}.

\subsection{Numerical simulations}
 \noindent In this subsection, we demonstrate the efficiency of the modified projected Gauss-Newton (MPG-N) method using several IEEE bus test cases from \cite{ICEG:13} (IEEE 14 bus, IEEE 39 bus, IEEE 57 bus and IEEE 118 bus). We chose an optimal point $x^* \in \textbf{C}$, then we generate $s_R = p(x^*)$ and $s_I = q(x^*)$ (see also \cite{CoWeZh:20}). We apply MPG-N method on problem \eqref{eq:opt_pr_f} and PGD method on  problem \eqref{eq:op-form1}, where $F$ is given in \eqref{eq:objft}, and test whether the algorithms can reach $x^*$ from a random feasible starting point. The stopping criterion for both algorithms is $\|F(x_k)\|\leq 10^{-3}$. The results are given in Figure \eqref{fig:my_label1}, where we plot the evolution of the function value $\|F(x_k)\|$ along iterations. From this figure one can observe that in the beginning, the PGD performs better than the MPG-N method. However, MPG-N method requires small number of iterations (even 5 times less) than the PGD in order to achieving the desired accuracy.
\begin{figure}
    \centering
    \includegraphics[height = 2.9cm]{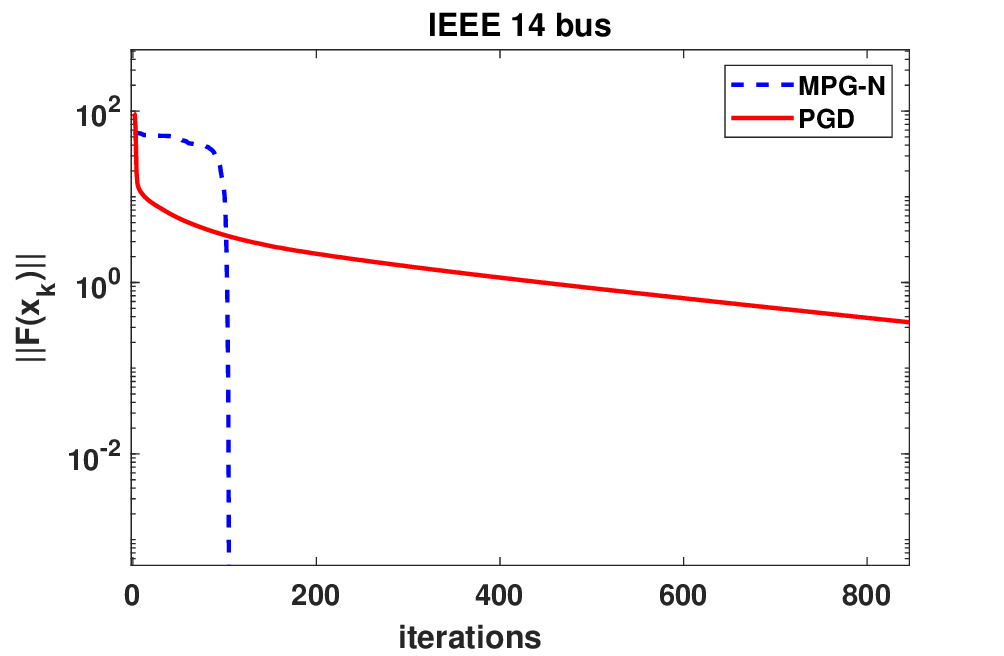}%scale=0.3 width = 4.2cm,
    \includegraphics[height = 3cm]{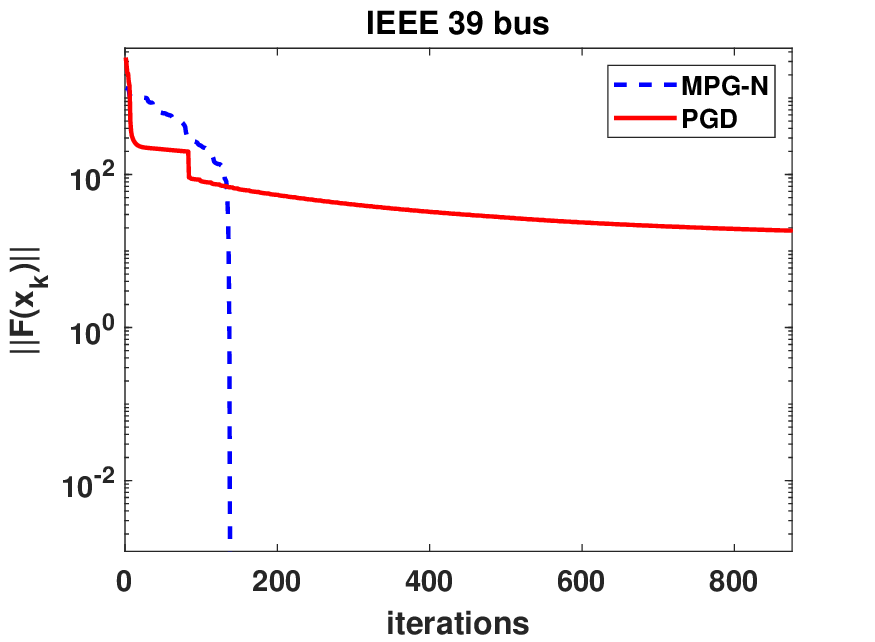}
    \includegraphics[height = 3.1cm]{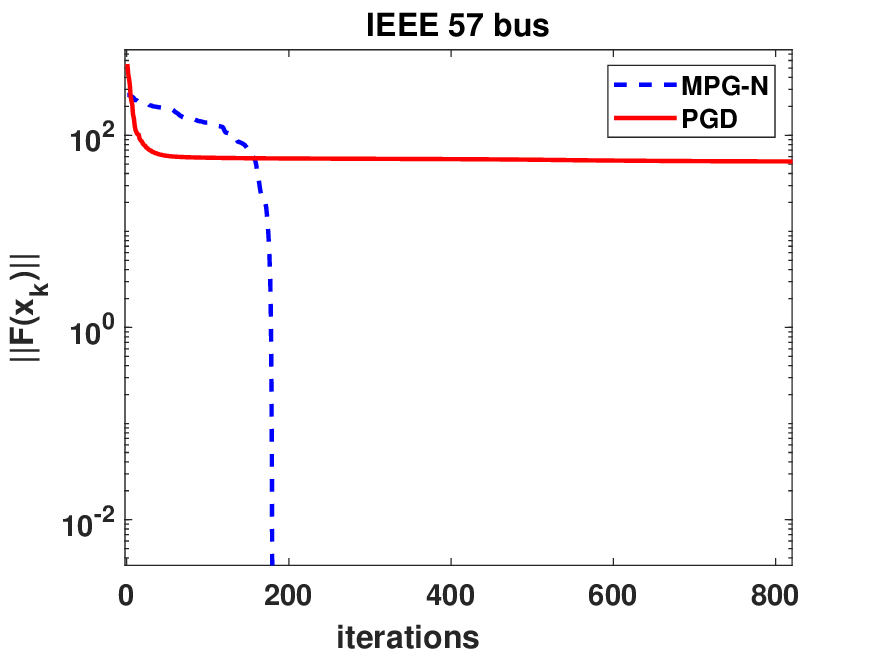}
    \includegraphics[height = 3.1cm]{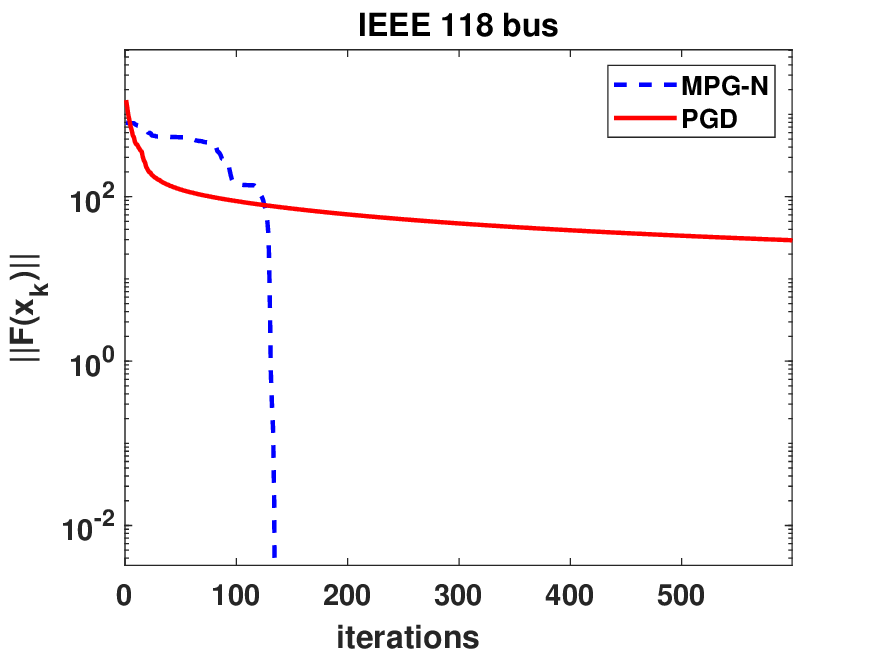}
    \caption{Comparison between MPG-N and PGD methods in terms of $\|F(x)\|$ along iterations on several IEEE bus systems.}
    \label{fig:my_label1}
\end{figure}

\section{Conclusion}
\noindent In this paper, we have proposed  a modified projected Gauss-Newton (MPG-N) method for solving constrained least-squares problems.  Under mild assumptions, we have proved global convergence results for the iterates. More precisely, we have proved that any limit point of the sequence generated by MPG-N algorithm is a stationary point and under the KL property, we have derived convergence rates in function values depending on the KL parameter. Finally, we have considered solving a power flow problem and compared the performance of our scheme with the projected gradient method, showing the efficiency of the proposed method on several IEEE bus test cases.
\vspace{1cm}

%%%%%%%%%%%%%%%%%%%%%%%%%%%%%%%%%%%%%%%%%%%%%%%%%%%%%%%%%%%%%%%%%%%%%%%%%%%%%%%%
%\section*{APPENDIX}

\section*{ACKNOWLEDGMENT}
\noindent The research leading to these results has received funding from:  ITN-ETN project TraDE-OPT funded by the European Union’s Horizon 2020 Research and Innovation Programme under the Marie Skolodowska-Curie grant agreement No. 861137;  NO Grants 2014-2021, RO-NO-2019-0184, under project ELO-Hyp, contract no. 24/2020;  UEFISCDI PN-III-P4-PCE-2021-0720, under project L2O-MOC, nr. 70/2022.

%%%%%%%%%%%%%%%%%%%%%%%%%%%%%%%%%%%%%%%%%%%%%%%%%%%%%%%%%%%%%%%%%%%%%%%%%%%%%%%%
%\textcolor{red}{
%\section{To Do:}
%Consider more generlal problem:
%\begin{align}
%    \min_{x} \|F(x)\| + h(x),
%\end{align}
%where $h$ is proper lower semicontinuous and convex ($\|\cdot\|_1$). 
%\begin{itemize}
 %   \item Kl convergence in sequences.
  %  \item If a subsequence $x_{k_j} \to x^*$, can we prove that $F(x^*) = 0$ and satisfies the property of the function $h$ (sparse or feasible) ?
%\end{itemize}
%}

\end{document}